\documentclass{article}
\usepackage{amsmath,amssymb,graphicx,xypic}
\usepackage[round]{natbib}

\usepackage{breakcites}

\usepackage[breaklinks]{hyperref}

\usepackage{color}

\newcommand{\colb}{\textcolor{black}}
\newcommand{\colr}{\textcolor{black}}

\newcommand{\QED}{\hspace*{\fill}\rule{2.5mm}{2.5mm}}
\newtheorem{theorem}{Theorem}[section]
\newenvironment{proof}{\noindent{\bf Proof\ }}{\QED\\}
\newcommand{\R}{\mathbb{R}}

\newcommand{\argmin}{argmin}

\begin{document}
\begin{center}
\vspace{0.5cm} {\Large \bf An objective look at obtaining the plotting positions for QQ-plots}\\
\vspace{0.4cm} Reza Hosseini$^a$, Akimichi Takemura$^b$,\\
$^a$IBM Research Collaboratory, Singapore\\
$^b$Department of Mathematical Informatics, University of Tokyo, Japan\\
$^a$rezah@sg.ibm.com
\end{center}

\begin{abstract}
Choosing the plotting positions for the QQ-plot has been a subject of much debate in the statistical and
engineering literature. This paper looks at this problem objectively by considering three frameworks:
distribution-theoretic; decision-theoretic; game-theoretic. In each framework, we derive the plotting positions
and show that there are more than one legitimate solution depending on the practitioner's objective. This work
clarifies the choice of the plotting positions by allowing one to easily find the mathematical equivalent of
their view and choose the corresponding solution. This work also discusses approximations to the plotting positions when
no closed form is available.\\
\vskip 1mm
\noindent {\bf Key Words}:  Plotting position; Loss function; Invariant; QQ-plot; Distribution-free
\end{abstract}

\section{Introduction}

A quantile-quantile plot (QQ-plot) is a graphical method for comparing observed data with a proposed (estimated)
distribution. Often the order statistics of the data are compared with the quantiles of a distribution which is
fitted to the data. For example a normal distribution can be \colr{fitted} to observed independent data: $x_1,\cdots,x_n,$
using maximum likelihood method: $\hat{X} \sim N(\hat{\mu},\hat{\sigma}^2)$ and then to check the goodness of the fit,
one can plot the order statistics, $x_{(1)},\cdots, x_{(n)}$ (data arranged in non-decreasing order), versus the quantiles of 
the estimated distribution
$\hat{X}$: $(q_{\hat{X}}(p_1),\cdots,q_{\hat{X}}(p_n))$ for an appropriate choice of probabilities $0< p_1
\leq p_2 \leq \cdots \leq p_n<1$, which we call a {\it probability index vector} (PIV). The choice of the PIV is the so
called {\it plotting positions problem}. In fact we can view this problem more generally by plotting
$x_{(1)},\cdots, x_{(n)},$ versus $f_1(\hat{X}),\cdots,f_n(\hat{X})$, where $f_1,\cdots,f_n$ are functions of the
distribution \colr{of} $\hat{X}$ -- without mentioning or using quantiles. However one can always apply the cumulative
distribution function (CDF) of $\hat{X}$, $F_{\hat{X}}$ to $(f_1(\hat{X}),\cdots,f_n(\hat{X}))$ to get back a
PIV, $(p_1,\cdots,p_n)$, and therefore the problem can be viewed as finding the appropriate PIV at least
in the continuous case, \colr{which is the case we consider here}.

The plotting positions problem has received significant attention in the literature.
The approaches suggested can be divided \colr{into} two general cases: (1) distribution-free methods; (2)
distribution-dependent methods. In (1) PIV does not depend on the underlying distribution of the data and in (2)
one utilizes some assumptions regarding the distribution of the data in finding the plotting positions. One may
argue that in (2) if the shape of the distribution is known, a QQ-plot \colr{is not useful}. This is not the
case because: we may have only partial information about the shape of the distribution, for example we may know
that the underlying distribution is a Gamma distribution and like to compare the data to a fit from the
Exponential distribution (a special case of Gamma distribution); even in the case that the assumed distribution and the fitted
distribution are the same, the QQ-plot can tell us how well the estimated distribution is performing across the
range of the quantiles. In this paper, we mainly focus on solving (1) but will clarify the difference in our
discussion.

\cite{weibull-1939} originally suggested $p_i={i}/{(n+1)}$ (Weibull method) and many other authors such as
\cite{yu-2001} and \cite{harter-1984} addressed this problem. \cite{makkonen-2008} provided evidence for the
Weibull method  by showing that the probability of non-exceedance of a new sample, denoted by $X_f$, from $X_{(i)}$ is exactly equal
to ${i}/{(n+1)}$: $P(X_f\leq X_{(i)})=i/(n+1)$. However some other authors (e.g.\;\cite{lozano-2013}) provide evidence for another intuitively appealing,
distribution-free solution to this problem which consists of the median of the Beta distribution of the CDF
applied to the order statistics: $Med\{F_X(X_{(i)})\}=Med\{Beta(i,n+1-i)\}$ (discussed in \cite{beard-1943},
\cite{benard-1953}, \cite{folland-2002}, \cite{hosseini-2009-thesis}). We call this method, the \colb{{\it Beta Median} (BM) method}.

The contributions of this work are as follows.  (1) It systematically defines intuitively appealing and \colr{rigorous} mathematical objectives to compare
the plotting positions schemes and \colr{in each case finds} the optimal plotting positions from three points of view: distribution-theoretic; decision-theoretic; game-theoretic. Most of the other relevant works consider a \colr{distribution-theoretic} approach and obtain different solutions from various assumptions 
(e.g.\;\cite{makkonen-2008} and \cite{lozano-2013}). These solutions are compared here. The decision-theoretic and game-theoretic views are new to the best of our knowledge.  In the decision-theoretic
framework, we find the optimal positions by minimizing appropriate loss functions defined for quantiles. In the
game-theoretic framework, we define intuitively appealing games in which two players choose plotting positions and
after the true distribution (or a future sample) is revealed, they exchange money based on the outcome. 
(2) This work provides a careful investigation of the accuracy of several approximations to the median of \colr{the Beta} distribution, specifically developed
for finding the optimal plotting positions (mostly of the form $(i-a)/(n+b)$ for some constants $a,b$, e.g.\;discussed in \cite{cunnane-1978}) along the PIV ($p_1,\cdots,p_n$) for small and large $n$.  This work also provides a comprehensive  comparison between the two popular solutions of the plotting positions problem (based on the expectation and median of the corresponding Beta distribution) for arbitrary sample sizes in terms of their difference and 
fraction.

Section \ref{sect:frameworks} considers three mathematical frameworks to address this problem: a
distribution-theoretic framework; a decision-theoretic framework; a game-theoretic framework. The
decision-theoretic and game-theoretic frameworks can guide a practitioner \colr{to choose the appropriate
 PIV}. In each of these frameworks, we find the optimal PIV or state that it is not possible to find
a distribution-free solution. In the distribution-theoretic approach, we consider the distribution of either the
order statistics or their CDF and analyze what quantities \colr{of} the estimated distribution match these random
values \colr{better}. The decision-theoretic framework uses appropriate loss functions for quantiles which are invariant under
monotonic transformations of data (and distributions) to pick the optimal PIV. In the game-theoretic approach, we
envision some gambling games in which players pick a PIV in order to compete for monetary values
and the solution to the plotting positions problem is deduced from the optimal strategy of the player. Section
\ref{sect:approx} discusses some approximations to the Beta Median method and shows that for an approximation of
the popular form $(i-a)/(n+1-2a)$, to perform well at $p_1, p_n$, when $n$ becomes large, it is necessary to have
$a=\log(e/2)$ (where $\log$ denotes natural logarithm). However for such choice of $a$, this approximation is not very accurate in some percentiles, for example around 
the $10th$ percentile of the
data. We show that the algorithm developed by \cite{cran-1977} performs well across the data range. This algorithm is
implemented in R for calculating the quantiles of a given Beta distribution but is not the default for
calculating plotting positions. Section \ref{sect:comparison} compares the PIV of Weibull method to the PIV of Beta Median method. We make a comparison between the plotting positions both in terms of their difference and ratio. The plotting positions of the two methods are always close in terms of their difference. They are also  
 close in terms of their ratio toward the center of PIV, but they differ at the two ends of the PIV.

\section{Choosing the plotting positions} \label{sect:qq-plot}
\label{sect:frameworks}

First we introduce notation for the rest of the paper. Let $X_1,\cdots,X_n$ be a univariate independent
identically distributed  (i.i.d) sample from a continuous random variable $X$ with CDF $F_X$.
We denote the quantile function (the inverse of $F_X$) by $q_X$ and the non-decreasingly ordered sample by
$X_{(1)},\cdots,X_{(n)}$. Then we define $U_i=F_X(X_i),\;i=1,2,\cdots,n$ to be the (random) probability of
non-exceedance from each sample point. Note that the $U_i$ sequence is i.i.d with uniform distribution on [0,1] and
denote the ordered sample by $U_{(i)}$. Since $F_X$ is non-decreasing, we have, $U_{(i)}=F_X(X_{(i)})$. It can be shown
that $U_{(i)}$ follows a Beta distribution (e.g.\;\cite{folland-2002}) with density function:
\begin{equation}
f_{\alpha,\beta}(x)=\frac{\Gamma(\alpha+\beta)}{\Gamma(\alpha)\Gamma(\beta)}x^{\alpha-1}(1-x)^{\beta-1},
\label{eqn-beta-density}
\end{equation}
for $\alpha=i$ and $\beta=n+1-i$. We denote this distribution by $Beta(i,n+1-i)$. In the following we find the
optimal plotting positions using three frameworks: distribution-theoretic approach; decision-theoretic approach;
game-theoretic approach.


\subsection{Distribution-theoretic approach}
Here we use distribution theory to find the plotting positions. Suppose we intend to \colr{create} a QQ-plot of the order
statistics with respect to a proposed distribution $\hat{X}$ with CDF $F_{\hat{X}}$. In order to find the
location of plotting positions for the order statistics $X_{(i)}$, we can consider the solution to be any of
the following:
\begin{itemize}
\item The plotting position for $X_{(i)}$ should be the value of $F_X$ at the average of $X_{(i)}$
in an appropriate sense, for example the expectation or the median:
\[p_i^{E}:={F_X}(E\{X_{(i)}\}),\;\;\mbox{or}\;\;p_i^M:={F_X}(q_{X_{(i)}}(1/2)).\]
Then $p_i^{E}$ is not distribution-free. In order to show that, consider the special case of a random sample with only one element ($n=i=1$) and let $X$ be distributed as 
exponential distribution with parameter $\lambda$ and CDF $F_X(x) = 1 - \exp(-\lambda x),\;x \geq 0$ . Then
\[p_1^{E}:= {F_X}(E\{X_{(1)}\}) = {F_X}(E\{X\}) = F_X(\lambda^{-1}) = 1-\exp(-\lambda^2),\]  
which depends on $\lambda$.
However $p_i^M$ is distribution-free because for 
a strictly increasing function $\phi:\R \rightarrow \R:$
\[F_{\phi(X)}(q_{\phi(X)_{(i)}}(1/2)) = F_{\phi(X)}(q_{\phi(X_{(i)})}(1/2) =\] \[F_{\phi(X)}(\phi(q_{X_{(i)}}(1/2))) =  F_X(q_{X_{(i)}}(1/2)). \]
Therefore in order to calculate $p_i^M$, we can assume $X$ is a uniform distribution on [0,1], which we denote by $U$. We conclude 
\[p_i^M = {F_U}(q_{U_{(i)}}(1/2)) = q_{U_{(i)}}(1/2),\]
and therefore $p_i^M$ is equal to the median of $Beta(i,n+1-i)$.  In fact \cite{lozano-2013} suggest using $p_i^M$ as an intuitively appealing formula for plotting positions and develop polynomial equations to find them. These polynomials were also developed in \cite{hosseini-2009-thesis} (pages 234 and 235) where the choice for the plotting positions is motivated by a decision-theoretic approach.

\item The plotting position for
$X_{(i)}$ should be the value in $F_{\hat{X}}$, corresponding to the average $U_{(i)}$ in an appropriate sense,
either the expectation or the median:
\[p_i^E:=E\{U_{(i)}\},\;\;\mbox{or}\;\;p_i^M:=q_{U_{(i)}}(1/2).\]
Since $U_{(i)} \sim Beta(i,n+1-i)$, $p_i^E=i/(n+1)$ and the median does not have closed form. Later we will discuss the
approximations to this median.
\item Consider a new unobserved sample $X_f$. Then the plotting position of $X_{(i)}$ should be the probability of
 non-exceedance of this new sample, $X_f$, from $X_{(i)}$: $P(X_f\leq X_{(i)}) = i/(n+1),$ which is the same as Weibull method. 
 This is essentially the content of
\cite{makkonen-2006} and \cite{makkonen-2008}.
\end{itemize}
\subsection{Decision-theoretic approach}
In this section, we take a decision-theoretic approach by considering appropriate loss functions for assigning
the plotting positions. For example, we can measure the loss by the absolute value loss $|X_{(i)}-q_X(p)|$.
 Since it is a random quantity, we  minimize
the expected loss:
\[p_i=\underset{p \in [0,1]}{argmin}\; E \{|X_{(i)}-q_X(p)|\}.\]
Again the solution to this is not distribution-free since a loss function defined as such is not
distribution-free (or equivalently it is not invariant under strictly monotonic distributions).

 \cite{hosseini-2009-thesis} and \cite{hosseini-2010-PLF} introduced loss functions for quantiles
which are invariant under strictly monotonic functions. The {\it Probability Loss (PL) function}  corresponding to
a random variable $X$ with distribution function $F_X$ is defined to be
\begin{equation}
\delta_X(a,b):=P(a<X<b)+P(b<X<a),
\label{eqn-PL-def}
\end{equation}
which is simply equal to $|F_X(b)-F_X(a)|$ if the distribution is continuous. Also note that $\delta_X$ only depends 
on the random variable $X$ through its distribution $F_X$. Therefore we can apply this definition to a distribution function $F_X$ and denote it by $\delta_{F_X}$. We can also apply this loss to
functions of data, $D_1,D_2$ (e.g.\;$D_1=X_{(3)},D_2=X_{(4)}$): $\delta_X(D_1,D_2)=|F_X(D_2)-F_X(D_1)|.$ \cite{hosseini-2010-PLF}
showed many desirable properties of this loss function, in particular its invariance under strictly monotonic
transformations of the real numbers: Let $\phi: \R \rightarrow \R,$ be a strictly monotonic function, then PL
satisfies the following invariance:
\begin{equation}
\delta_X(a,b)=\delta_{\phi(X)}(\phi(a),\phi(b)).
\label{eqn-PLF-inv}
\end{equation}
Moreover, this loss is symmetric and it can be shown that it satisfies the triangle
inequality for the continuous variables (\cite{hosseini-2010-PLF}). Since this loss is random in
general, we need to consider a measure of average loss such as expectation.

 To obtain a distribution-free solution, we utilize
  the {\it Expected PL} (EPL):
 \begin{equation}
 p_i=\underset{p \in [0,1]}\argmin E\{\delta_{X} (X_{(i)},q_X(p))\}.
 \label{eqn-pi-EPL}
 \end{equation}
We also assume that the underlying distribution is continuous. Then the solution and its properties are given in
the following theorem.
\begin{theorem}
Suppose a random sample $X_1,\cdots,X_n$ from a continuous distribution function $F$ is given.
Then
\[E\{\delta_{X} (X_{(i)},q_X(p_i))\},\]
is minimized uniquely by $p_i=Med\{Beta(i,n+1-i)\}$. Moreover $p_i$ is the solution of the
equation $\sum_{j=i}^n  \binom{n}{j} u^j(1-u)^{n-j}=1/2,$ when solving for $u$, and we have $p_i=1-p_{n-i+1}$. 
\label{theo-QQ-plot}
\end{theorem}
\begin{proof}
For the first part, note that we want to minimize:
\[E\{\delta_{F_X}(X_{(i)},q(p))\}=E\{|F(X_{(i)})-p|\}=E\{|U_{(i)}-p|\}.\]
However $E\{|U_{(i)}-p|\}$ is minimized by choosing $p$ to be the median of $\{U_{(i)}\}$ and we know that $U_{(i)} \sim Beta(i,n+1-i)$.

To prove $p_i$ is  the solution of the equation $\sum_{j=i}^n  \binom{n}{j} u^j(1-u)^{n-j}=1/2,$ when solving for $u$, note that under the continuity assumption, we have
\[E\{\delta_{F_X} (X_{(i)},q_X(p_i))\}=E \{|F_X(X_{(i)})-p_i|\}.\]
Since $F_X$ is a continuous random variable, $F_X(X_{(i)})$ is also continuous. Hence the minimum is obtained by
solving $P(F_X(X_{(i)})\leq x)=1/2$. This is equivalent to $P(X_{(i)}\leq q_F(x))=1/2$.  It is well-known that (e.g.\;\cite{arnold-1992}) 
the distribution of the
order statistics, $X_{(i)}$ is given by
\[P(X_{(i)}\leq y)=\sum_{j=i}^n  \binom{n}{j} F_X(y)^j(1-F_X(y))^{n-j}.\]
Hence, the minimum is obtained by solving
\[\sum_{j=i}^n  \binom{n}{j} F_X(q_X(u))^j(1-F_X(q_X(u)))^{n-j}=\sum_{j=i}^n  \binom{n}{j} u^j(1-u)^{n-j}=1/2,\]
which does not have a closed form solution in general. Also note that the solution does not depend on $F_X$.
However, the solution always exists and is unique since $\sum_{j=i}^n \binom{n}{j} u^j(1-u)^{n-j}$ is monotonic
for each $i$, continuous on [0,1] and ranges between 0 and 1.

Finally the fact that the resulting PIV is symmetric follows  from  the  symmetry  of  the  Beta CDF as seen in Equation \ref{eqn-beta-density}.
\end{proof}

These equations can be solved for $n=1,2$. For $n=1$, we have $p_1=1/2$. For  $n=2$, we have
$p_1=1-{1}/{\sqrt{2}}$ and $p_2={1}/{\sqrt{2}}$. Note that for arbitrary $n$, the last equation is $x^n=1/2$.
Hence we have $p_n=1/\sqrt[n]{2}$ and $p_1=1-1/\sqrt[n]{2}$.

\cite{hosseini-2010-PLF} introduced a related non-random loss which does not require taking the expectations and
is more appropriate when the future sample is of interest. The {\it future-value probability loss} (FPL) for
functions of data $D_1,D_2$ (e.g.\;$D_1(X_1,\cdots,X_n)=X_{(1)},D_2(X_1,\cdots,X_n)=X_{(n)}$), is defined by
\begin{equation}
\gamma_X(D_1,D_2):= P (D_1<X_f<D_2) + P(D_2<X_f<D_1),
\label{eqn-FPL-def}
\end{equation}
where $X_f \sim X$ a new independent draw (future draw) of the random variable of interest. Again we can show the
desirable properties of this loss including an its invariance under strictly monotonic transformations:  If two statistics (functions of data), $D_1,D_2$, 
are {\it equivariant} under a strictly monotonic transformation $\phi$, i.e.
\[\phi(D_i(X_1,\cdots,X_n))=D_i(\phi(X_1),\cdots,\phi(X_n)),\;\;i=1,2,\]
then
\begin{equation}
\gamma_X(D_1,D_2)=\gamma_{\phi(X)} (\phi(D_1),\phi(D_2)).
\label{eqn-FPL-inv}
\end{equation}
\colr{Note that $\gamma_X$ is not random. Therefore, we let:}
 \[p_i=\underset{p \in [0,1]}\argmin\; \gamma_{X}(X_{(i)},q_X(p)).\]
\begin{theorem}
The plotting position for $X_{(i)}$ by minimizing $\gamma_X(X_{(i)},q_X(p))$ is given by
$p_i=Med\{Beta(i,n+1-i)\}.$ \label{theo-gamma}
\end{theorem}
\begin{proof}
We denote the CDF of $X_{(i)}$ by $F_i$ and proceed as follows.
\begin{align*}
\gamma_X(X_{(i)},q_X(p))&=P(X_{(i)}<X_f<q_X(p))+P(q_X(p)<X_f<X_{(i)})\\
&=\int_{-\infty}^{+\infty} P(x<X_f<q_X(p))d{F_{i}(x)}+\int_{-\infty}^{+\infty} P(q_X(p)<X_f<x)d{F_{i}(x)}\\
&=\int_{-\infty}^{q_X(p)} (p-F(x))d{F_{i}(x)}+\int_{q_X(p)}^{+\infty}(F(x)-p)d{F_{i}(x)}\\
&=\int_{-\infty}^{+\infty}|F(x)-p|d{F_{i}(x)} =E\{|F_i(X)-p|\},
\end{align*} 
which is again minimized at 
$p=Med\{Beta(i,n+1-i)\}$.
\end{proof}

\subsection{Game-theoretic approach}
Here we provide a game-theoretic approach for the plotting positions problem. Since we are mainly interested in
distribution-free results, we formulate the games in terms of $U_{(i)}$. However they can also be formulated in terms
of $X_{(i)}$ to get distribution-dependent results. This framework clarifies which
plotting positions scheme is the most appropriate for the given objective. We consider four games in which two
players A and B, use their method of picking the plotting positions for $U_{(i)}$, denoted by $p_i^A,p_i^B$
respectively and bet on the result. When the true distribution is revealed, A and B exchange money according to
one of the four following rules.
\begin{itemize}
\item {\bf Game (1):} When the true distribution is revealed, A wins if:
\[|U_{(i)}-p_i^A| < |U_{(i)}-p_i^B|,\]
in which case B pays one dollar to A and otherwise receives a dollar from A. The solution to this game is given
in the following theorem.\\
\begin{theorem}
For Game (1), $Med\{Beta(i,n+1-i)\}$ is the best strategy.
\end{theorem}
\begin{proof}
We claim that if A picks $p_i^A=Med\{Beta(i,n+1-i)\}$, he will insure for any $p_i^B \neq p_i^A$:
\[P\{|U_{(i)}-p_i^A| < |U_{(i)}-p_i^B|\}>1/2.\]
To show that, note that $p_i^A,p_i^B \geq 0$ and we have:
\begin{eqnarray*}
&& P\{|U_{(i)}-p_i^A| < |U_{(i)}-p_i^B|\}>1/2\\
&\Leftrightarrow& P\{(U_{(i)}-p_i^A)^2 < (U_{(i)}-p_i^B)^2\}>1/2\\
&\Leftrightarrow& P\{U_{(i)}(p_i^B-p_i^A)<(p_i^B-p_i^A)\frac{p_i^A+p_i^B}{2}\}>1/2\\
&\Leftrightarrow& \begin{cases} P \{U_{(i)}>\frac{p_i^A+p_i^B}{2}\}>1/2, & p_i^B<p_i^A\\P \{U_{(i)}<\frac{p_i^A+p_i^B}{2}\}>1/2, & p_i^B>p_i^A.\\
\end{cases}
\end{eqnarray*}
The last equation holds if and only if $p_i^A$ is the median of $U_{(i)}$: \[Med\{Beta(i,n+1-i)\}.\] 
\end{proof}
\item {\bf Game (2):} When the true distribution is revealed, A pays \[|U_{(i)}-p_i^A|-|U_{(i)}-p_i^B|,\] dollars to
B. Note that a negative value means A receives the magnitude of the value.
To find the best strategy for this game A needs to minimize $E\{|U_{(i)}-p_i^A|\},$ which is minimized again for
$q_{U_{(i)}}(1/2)$.
\item {\bf Game (3):} When the true distribution is revealed, A pays \[(U_{(i)}-p_i^A)^2-(U_{(i)}-p_i^B)^2,\] dollars to
B. Note that a negative value means: A receives the magnitude of the value.
To find the best strategy for this game, A needs to minimize $E\{(U_{(i)}-p_i^A)^2\},$ which is minimized for
$p_i=E\{U_{(i)}\}=i/(n+1)$.
\item  {\bf Game (4):} A and B play a game based on the result of a new sample, $X_f$, drawn from the true distribution.
A judge will keep sampling from the true distribution until $X_f$: falls between $q_X(p_i^A)$ and $X_{(i)}$, but not between 
$q_X(p_B^i)$ and $X_{(i)}$,  in
which case A pays one dollar to B; or falls between $q_X(p_B)$ and $X_{(i)}$, but not between $q_X(p_A^i)$ and $X_{(i)}$,  in which case B pays one dollar to A.
The player A can find the best strategy by minimizing for $p$ in $\gamma_{X}(X_{(i)},q_X(p))$ and the solution is again 
$Med\{Beta(i,n+1-i)\},$ as shown in Theorem \ref{theo-gamma}.
\end{itemize}

\section{Approximation of the plotting positions}
\label{sect:approx} We have seen that the desirable solution for the plotting position, $p_i$ in many frameworks
has turned out to be the median of the Beta distribution: $Beta(i,n+1-i),$ which is equal to the solution of
the equation
\begin{equation}
H(i,n):=\sum_{j=i}^n  \binom{n}{j} x^j(1-x)^{n-j}=1/2,
\label{eqn-H}
\end{equation}
\colb{and also appears in  \cite{hosseini-2009-thesis} and \cite{lozano-2013}. By symmetry of PIV (Theorem \ref{theo-QQ-plot}), we need to solve either $H(i,n),\;i\leq n/2$ or $H(i,n),\;i\geq n/2$. 
Since the polynomials $H(i,n),\;i\geq n/2$ are all increasing in $x$, a bisection
method can be used to find the solution with any desired accuracy. However, this method 
becomes slow for large $n$ because of the binomial coefficients
calculation. Fortunately, a fast algorithm for calculating the quantiles of the Beta
distribution is developed by \cite{cran-1977} which is implemented in C and R languages}.  Below we compare the
values obtained from this algorithm with the exact values as well as some other popular approximations suggested
in the literature. \colb{Explicitly note that the solution obtained through the bisection algorithm will be called Exact method since it can be found with any desired accuracy (i.e., the length of the interval obtained at the last iteration).} Despite the existence of good algorithms to approximate the median of Beta distribution, many authors and
packages use simpler formulas for approximating the median for the QQ-plot purpose.

There are various measures of the form $\colb{p_i={(i-a)}/{(n+b)}}$ suggested in the literature to
approximate the median of the corresponding Beta distribution which is the same as the solution to $H(i,n)=1/2$.
Due to the existence of more exact solutions for the median such as the one developed by \cite{cran-1977} (Cran's
Method), there is little practical need for solutions of the form $\colb{p_i={(i-a)}/{(n+b)}}$. \colb{However, because of the popularity of
these solutions and their implementation in commercial software, it is interesting to
compare the accuracy of different choices for $a$ and $b$. In particular we require $a$ and $b$ to
satisfy the following constraints:}
 
\begin{itemize}
\item[(1)] $p_{i}=\colb{{(i-a)}/{(n+b)}}$ should be equal to 1/2, when $n$ is odd and $i=(n+1)/2$.\\
Let $n=2k+1$ then by assuming above we have $\colb{{(k+1-a)}/{(2k+1+b)}}=1/2 \Leftrightarrow \colb{b=1-2a}.$
This condition is also stated in \cite{Erto-2013}.
\item[(2)] Symmetry: $p_i=1-p_{n-i+1}$.\\
From above we conclude:
\[\frac{i-a}{n+b}=1-\frac{n-i+1-a}{n+b} \Leftrightarrow b=1-2a,\]
which is the same as the requirement for holding (1). This is equivalent to Postulate 4 in \cite{Erto-2009}.
\item[(3)] $p_1=1-p_n$ should behave the same as $1-(1/2)^{1/n}$ in the limit:
\[\lim_{n \rightarrow \infty} \frac{1-(1/2)^{1/n}}{\frac{1-a}{n+b}}=1.\]
Calculating the limit using the L'H\^opital's rule, we have:
\begin{equation}
\lim_{n \rightarrow \infty} \frac{1-(1/2)^{1/n}}{\frac{1-a}{n+b}}= \lim_{n \rightarrow \infty} \frac{-(-1/n^2(1/2)^{1/n}\log(1/2))}{-(1-a)/(n+b)^2}=
\frac{-\log(1/2)}{1-a},
\label{eqn-limit-p1}
\end{equation}  where $\log$ denotes the natural logarithm. Letting the limit equal to 1, we get $a=1-\log 2.$
\end{itemize}

Finding plotting positions which satisfy conditions such as the above are referred to as axiomatic approach in
\cite{Erto-2013}.  In the above, (1) and (2) \colb{imply that $b =1-2a$, which is the form widely suggested in the
literature}.  Also (3) suggests $a=1-\log 2.$ Combining this 
with (1), we get $p_i=(i-a)/(n+1-2a),\;\colb{a=\log(e/2) \approx 0.3068528}.$   \cite{Erto-2013} suggests using
$a=n+\frac{n-1}{2^{1/n}-2}$ which depends on $n$ and
 has the correct limit of $\log(e/2)$. This is not surprising because the derivation of \cite{Erto-2013} is
 equivalent to letting $p_n={(n-a)}/{(n+1-2a)}=(1/2)^{1/n}$.

In Table \ref{table:method-comparison}, we have calculated the log PIV for $n=2,\cdots,5$ using the Exact method;
the Cran method; Erto method; $a=\log(e/2)$ method; \colb{\cite{kerman-2011} method ($a=1/3$)}. By definition, all these methods are
symmetric and we only need to compare the lower or upper half of the PIV vector. We have chosen the lower half
because the difference among the methods are more significant on the log scale in the lower half. For an odd $n$,
the position of the middle value is the same theoretically \colr{for all methods} and equal to $\log(1/2)$ and therefore
we have omitted that value. By definition, the Erto method matches the tail values at $p_1,p_n$ and
therefore it is exact for $n=1,2,3$ as shown in the table. It starts to
deviate from the exact values for larger $n$ and for positions which are not very close to $i=1,n$ or $i=n/2$. 
We observe that the Cran method for all $n=2,\cdots,5$ is exact. Also $a=\log(e/2)$ method, which is asymptotically the same as Erto method,  performs close to
the exact values. \colb{Kerman method} performs poorly, especially for $p_1$ (and therefore $p_n$). This
poor performance of \colb{Kerman method} cannot be remedied even for large $n$, since in order to capture
the correct limit one has to let $a=\log(e/2)$.

Figure \ref{positions_methods_1_10.pdf} compares the approximated log positions to the exact log positions (grey
line) for larger sample sizes,
 $n=10,20,\cdots,150$ and for the positions 1 (Left Top Panel); $(1/10)n$ (Top Right Panel);
 $(2/10)n$ (Bottom Left Panel); $(4/10)n$  (Bottom Right Panel). The Cran method is given in dashed line; Erto method
 is given
 in dotted line; the $a=\log(e/2)$ method is given in filled circles; the $a=1/3$ method is given in triangles.
 Note that the sample sizes, $n,$ are chosen to be multiples of $10$ in order to get an integer position for all cases. In all
 panels we observe that the Cran method is almost exact. The Erto method (dotted) and  $a=\log(e/2)$ method
 (filled circles)
 are very close to
 exact value for $p_1$ and start to deviate from the exact value (larger than) for $(1/10)n$ and $(2/10)n$
 positions. On the contrary the $a=1/3$ method (triangles) is poor for $p_1$ and start to be more accurate for
 $(1/10)n$ and $(2/10)n$. Finally all methods start to be close to the exact as we move toward $n/2$ as seen in
 $(4/10)n$.

\begin{figure}
\centering
\includegraphics[width=0.35\textwidth,angle=270] {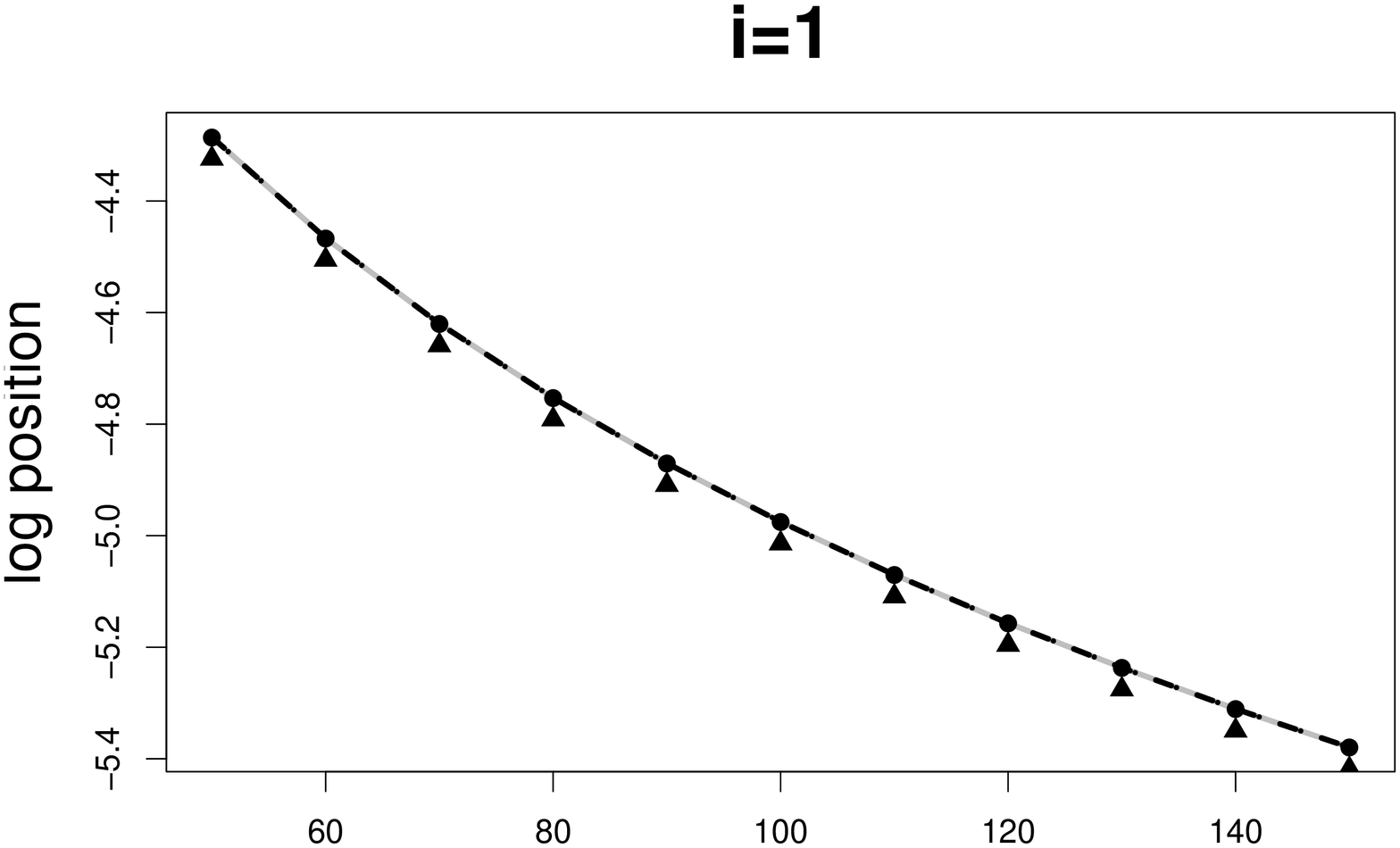}\includegraphics[width=0.35\textwidth,angle=270] {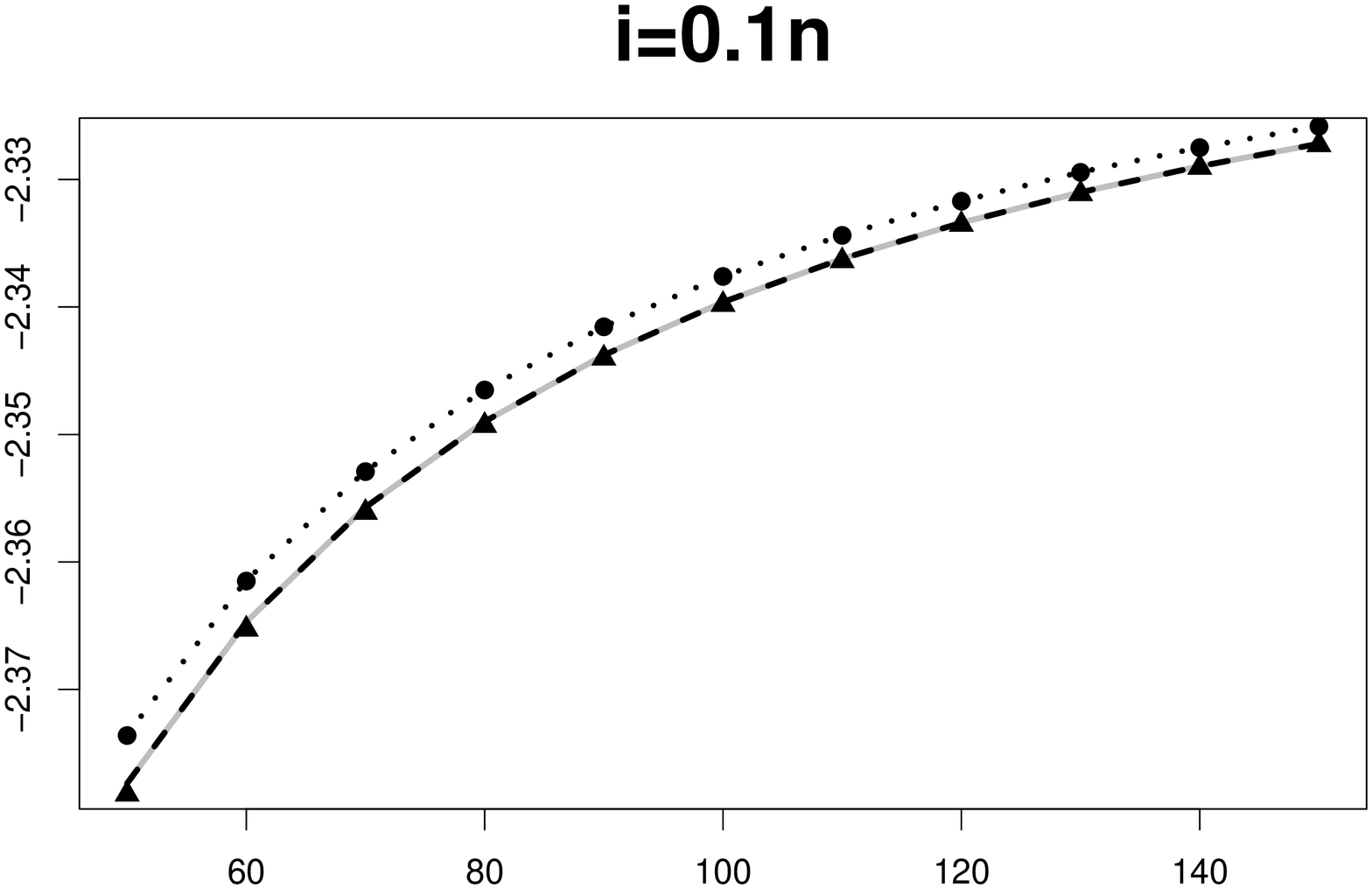}

\includegraphics[width=0.35\textwidth,angle=270] {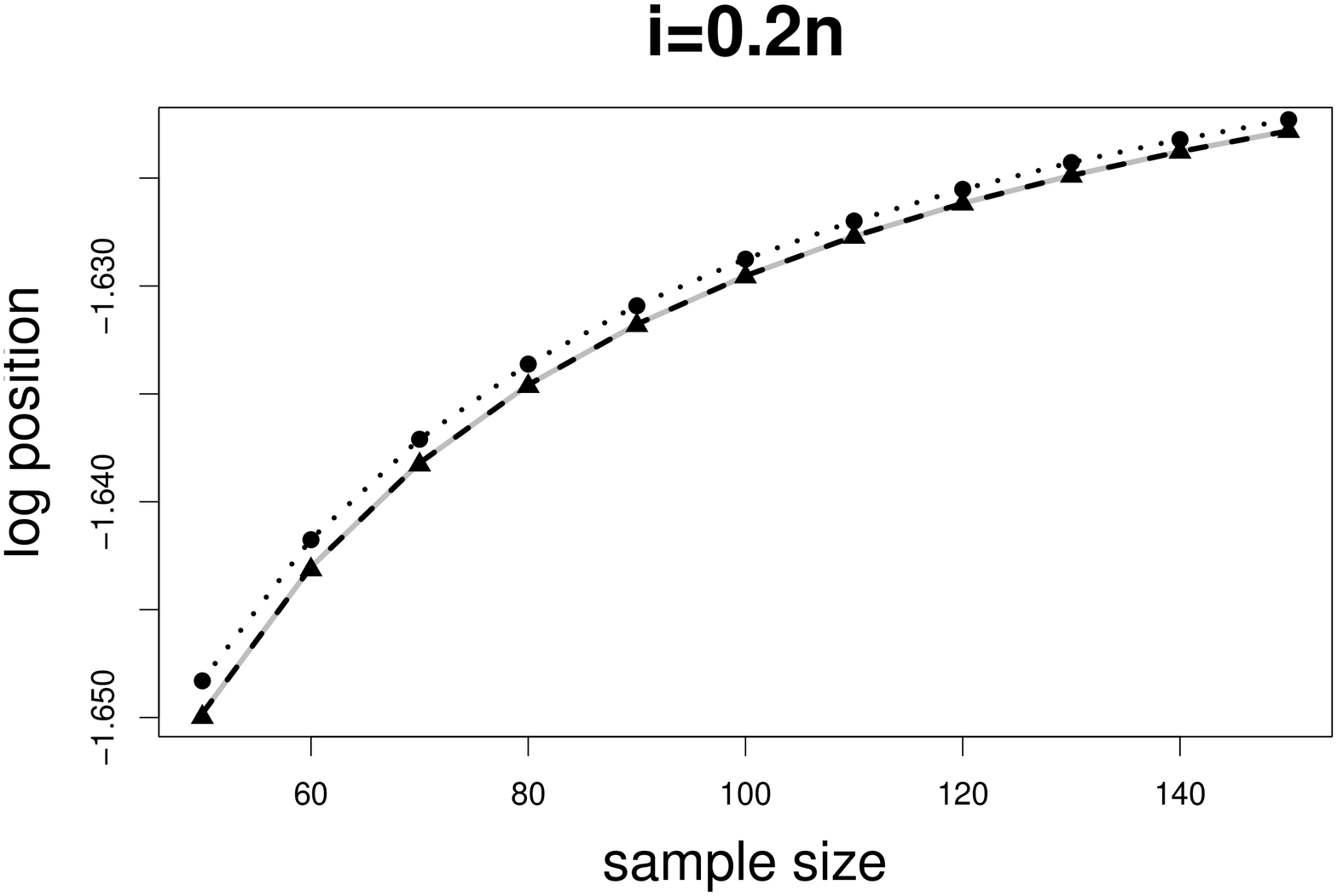}\includegraphics[width=0.35\textwidth,angle=270] {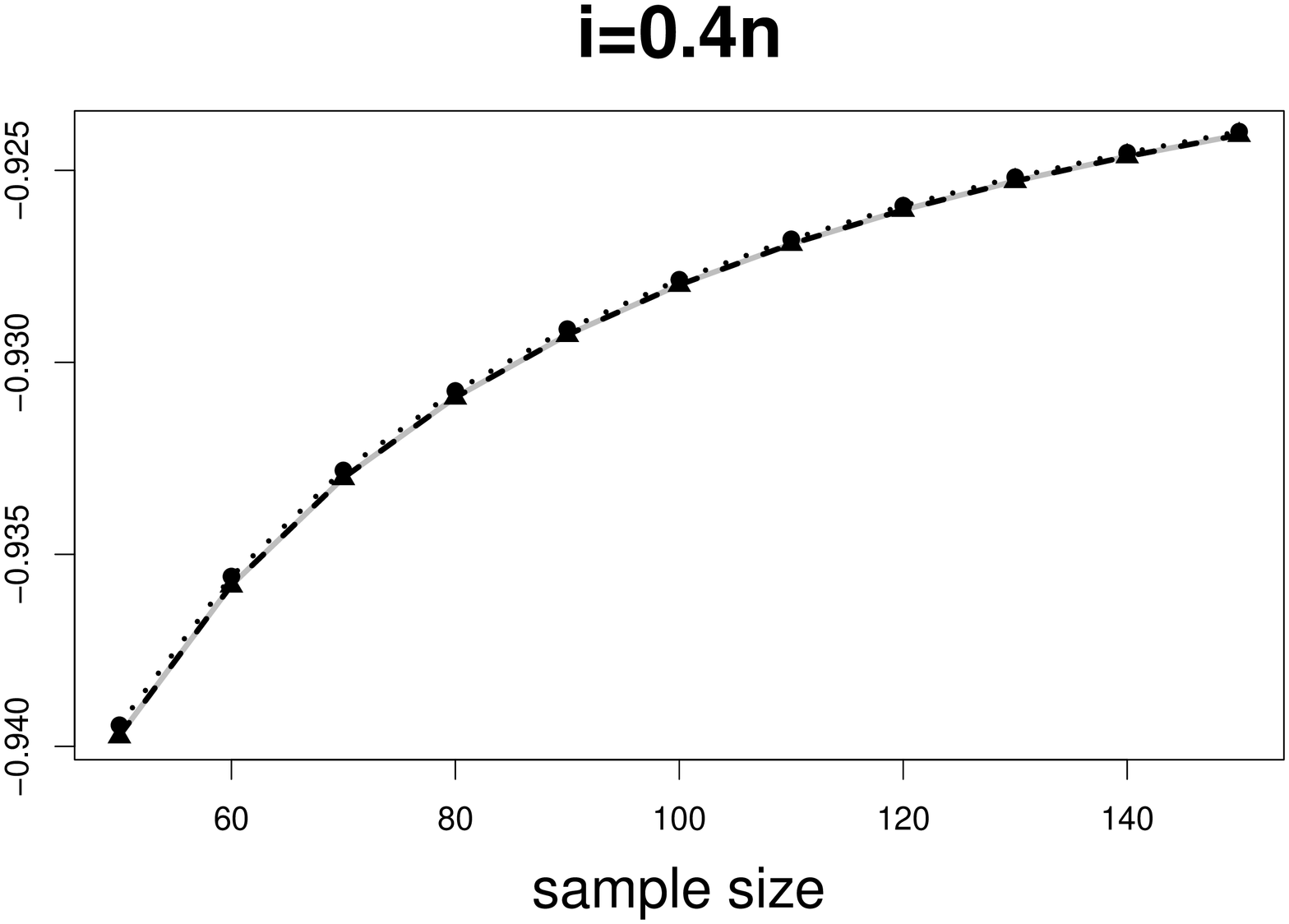}
\caption{\footnotesize Figure compares the approximated log positions to the exact log positions (grey line) for
 $n=50,60,\cdots,150$ and for the positions 1 (Left Top Panel); $(1/10)n$ (Top Right Panel);
 $(2/10)n$ (Bottom Left Panel); $(4/10)n$  (Bottom Right Panel). The Cran method is given in dashed line;
 Erto method is given
 in dotted line; the $a=\log(e/2)$ method is given in filled circles; \colb{Kerman method} is given in triangles.}
 \label{positions_methods_1_10.pdf}
\end{figure}

\begin{table}
\centering
\caption{Comparing the approximating methods for the natural logarithm of quantiles for $n=2,\cdots,5.$ The values different from
corresponding exact values are shown in bold.}
 \footnotesize
\begin{tabular}{|l|c|c|c|c|c|c|c}
\hline
\mbox{Method} & $n=2,\;(p_1)$ & $n=3,\; (p_1)$ & $n=4, (p_1,p_2)$ & $n=5, (p_1,p_2)$ \\
\hline
 Exact  & -1.228 &-1.578 &-1.838\; -0.9526 &-2.044\; -1.159 \\
Cran  & -1.228 &-1.578& -1.838\; -0.9526 &-2.044\; -1.159\\
Erto  &  -1.228 & -1.578&  -1.838\; {\bf -0.9510}  &-2.044\; {\bf -1.156} \\
$a=\log({e}/{2})$  & {\bf -1.236} &{\bf -1.586} &{\bf -1.845}\; {\bf -0.9519}& {\bf -2.050}\; {\bf -1.157}   \\
\colb{Kerman}  & {\bf -1.253} &{\bf -1.609}&{\bf -1.872}\; {\bf -0.9555} & {\bf -2.079}\; {\bf -1.163}\\
\hline
\end{tabular}
 \label{table:method-comparison}
 \end{table}

\section{Comparison}
\label{sect:comparison}
This subsection briefly compares the plotting positions from the Weibull method (WM) and the \colb{Beta Median (BM) method}. We denote their corresponding PIV by PIV$^W=(p_1^W,\cdots,p_n^W)$ and PIV$^B=(p_1^B,\cdots,p_n^B)$.
Figure \ref{QQ_plot_PL_Weibull.pdf} compares  BM to the WM for $n=2,3,4,5$, showing that BM chooses larger values
for $p_i,\;i>n/2$ and smaller values for $p_i,\;i<n/2$ as compared to WM. The difference is most
noticeable at $p_1$ and $p_n$ where quantiles closer to the tails of the distribution are compared to $X_{(1)}$
and $X_{(n)}$. 
\begin{figure}
\centering
\includegraphics[width=0.9\textwidth] {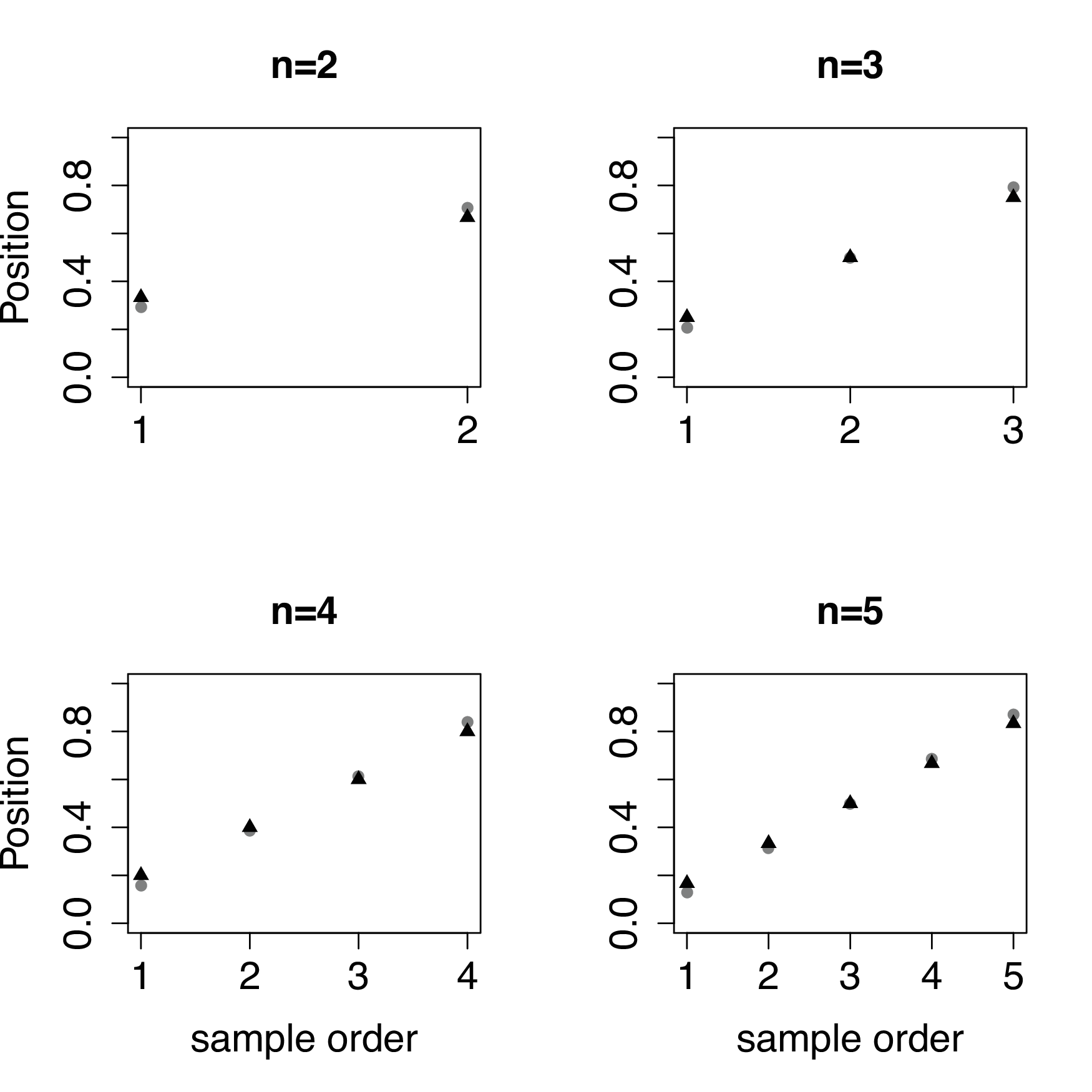}
 \caption{\footnotesize The Weibull quantile positions (black triangle) are compared to Beta Median positions
 (grey circles) for
 $n=2,\cdots,5$.}
 \label{QQ_plot_PL_Weibull.pdf}
\end{figure}
To compare WM and BM for arbitrary $n$, we use a Theorem of \cite{payton-1989} on the difference of the mean and median of 
the Beta distribution, which has not received the attention it deserves in the plotting positions literature. This theorem states that if $X \sim Beta(\alpha,\beta),\;\alpha,\beta>1$, then 
\begin{eqnarray}
0<E\{X\}-q_X(1/2)<|\alpha-\beta|/\{(\alpha+\beta)(\alpha+\beta-2)\},\; & \alpha<\beta&  \nonumber\\
0<q_X(1/2)-E\{X\}<|\alpha-\beta|/\{(\alpha+\beta)(\alpha+\beta-2)\},\; & \alpha>\beta&  \nonumber\\
q_X(1/2)=E\{X\},\; & \alpha=\beta,&
\label{payton-theorem}
\end{eqnarray}
where we have corrected line 2 of the statement (changed $0<E\{X\}-q_X(1/2)$ to $0<q_X(1/2)-E\{X\}$).

In the following, we present a theorem which clarifies the relationship between the WM and BM. We compare the plotting positions in terms of their difference and their ratio. The theorem provides a bound for the difference of $p_i^W$ and $p_i^M,$ which  is smaller for central positions and smaller than $1/(n+1)$ for all positions. Also we find the bound of $1/i$ for the ratio of $p_i^B$ and $p_i^W$. Since $p_i^B$ and $p_i^W$ are larger than 1/2 for $i>n/2$, we compare also the ratio of  $1-p_i^B$ and $1-p_i^W$ and show it is bounded by $1/(n+1-i)$. In summary, for the first half of the PIV, when we move away from $i=1$ toward the center,
the ratio becomes small rapidly. For the second half of PIV, when we move from $i=n$ toward the center, the ratio of $1-p_i^B$ and $1-p_i^W$ becomes small rapidly. However the ratio of $p_i^B$ and $p_i^W$, at the beginning of PIV (in particular for $i=1$) and the ratio of $1-p_i^B$ and $1-p_i^W$ at the end of PIV (in particular for $i=n$), are not close to 1 (even for large $n$). In fact we show that both ratios are asymptotically equal to $\log(2) \approx 0.69.$ The exact statements for these claims are given in the following theorem, in which $[m]$ denotes the largest integer less than or equal to $m$.

\begin{theorem} Assume $n>2$ is a natural number. Also let 
\[PIV^W=(p_1^W,\cdots,p_n^W), \mbox{ and } PIV^B=(p_1^B,\cdots,p_n^B),\] denote the PIV for the Weibull method (WM) and Beta Median (BM) method respectively. Then the following results hold.
\begin{itemize}
\item[(a)] $0<p_i^W-p_i^B<|n+1-2i|/(n^2-1),\;i=1,\cdots,[n/2].$
\item[(b)] $0<p_i^B-p_i^W<|n+1-2i|/(n^2-1),\;i=[n/2]+1,\cdots,n.$
\item[(c)] $p_i^W=p_i^B,\;i=(n+1)/2,$ if $n$ is odd.
\item[(d)] $|p_i^W-p_i^B|<1/(n+1),\;i=1,\cdots,n.$
\item[(e)]  $0<p_1^B<p_1^W<1$ and $0<p_n^W<p_n^B<1.$  
\item[(f)] $p_{i-1}^W<p_i^B<p_{i}^W,\;i=2,\cdots,[n/2]$
\item[(g)] $p_{i-1}^B<p_i^W<p_{i}^B,\;i=[n/2]+1,\cdots,(n-1)$
\item[(h)] $|p_i^B/p_i^W-1| < |n+1-2i|/\{i(n-1)\} < 1/i,\;i=1,\cdots,n.$
\item[(i)] $|(1-p_i^B)/(1-p_i^W)-1| < |n+1-2i|/\{(n+1-i)(n-1)\} < 1/(n-i),\;i=1,\cdots,n.$
\item[(j)] $\underset{n \rightarrow \infty}{\lim} p_1^B/p_1^W=\underset{n \rightarrow \infty}{\lim} (1-p_n^B)/(1-p_n^W)=\log(2).$
\end{itemize}
\end{theorem}
\begin{proof}
For the proof, we apply Equation \ref{payton-theorem} (\cite{payton-1989}) to $\alpha=i$ and $\beta=n+1-i$ and use the fact that $p_i^B=Med\{Beta(i,n+1-i)\}$ and $p_i^B=Mean\{Beta(i,n+1-i)\}$.
\begin{itemize}
\item[(a)] Since $i=2,\cdots,[n/2]$, we have $i=\alpha \neq \beta=n-1$. Then we apply Equation \ref{payton-theorem} (line 1).
The case $i=1$ needs a special treatment because $\alpha=1$. In this case $|n+1-2i|/(n^2-1)=1/(n+1)$.
Since $p_1^W=1/(n+1)$ we only need to show $p_1^B<1/(n+1)$ for $n\geq3$. However in this case $\alpha=1,\beta=n$ and the density function of $X \sim Beta(\alpha,\beta)$ is decreasing. This will conclude the median is smaller than the mean by various methods. For example by noting that the median=$\underset{\mu}{\inf}\;E\{|X-\mu|\}$ and mean=$\underset{\mu}{\inf}\;E\{|X-\mu|^2\}$.
\item[(b)] This follows from the above and the Beta distribution symmetry.
\item[(c)] Since $i=(n+1)/2,$  $\alpha=\beta=n+1-i=(n+1)/2$ and and we apply Equation \ref{payton-theorem} (line 3).
\item[(d)] This follows from the above by noting that  $|n+1-2i|/(n^2-1) \leq 1/(n+1),\;i=1,\cdots,n.$
\item[(e)] Straight forward from above.
\item[(f)] This follows from $(p_i^W-p_i^B)<1/(n+1),\;i=1,\cdots,[n/2]$ and noting that $p_i^W=i/(n+1)$, which implies in $PIV^W$ two consecutive elements differ exactly by $1/(n+1)$.
\item[(g)] This follows from the above and the symmetry of Beta distribution.
\item[(h)]  We showed that $p_i^B=p_i^W+R_i$ where $|R_i| < |n+1-2i|/(n^2-1)$. Thus
\[|p_i^B/p_i^W - 1| = |(p_i^W+R_i)/p_i^W - 1| = \frac{|R_i|}{i/(n+1)} <  |n+1-2i|/\{i(n-1)\} < 1/i.\]
\item[(i)] This follows from above and the symmetry of PIV.
\item[(j)] This follows from the limit argument given in Equation \ref{eqn-limit-p1} for $a=0,\;b=1$, in which case $(1+a)/(n+b)=1/(n+1)=p_1^W.$ 
\end{itemize}
\end{proof}

\section{Discussion}

This work investigated the plotting positions problem using various frameworks:\\ distribution-theoretic; decision-theoretic; game-theoretic. While there has been a lot of previous work in this area -- each suggesting a different
formula for the plotting positions -- the validity and the assumptions under which any of these formulas are valid
were \colr{not clear}. This work addresses this issue by deriving the distribution-free plotting
positions under various reasonable objectives which are understandable by practitioners.

Two solutions which came out of the analysis in the above frameworks were the Weibull (expectation of Beta) and the
Beta Median methods. Despite the popularity of the Weibull method (e.g.\;\cite{makkonen-2008}), we showed that
it is not the only correct solution to this problem and the Beta Median method is the optimal under various
scenarios -- for example  to minimize the probability loss function (PL) or to have more chance to win a game in which 
the winner picks the closest quantile from the true distribution to the given order statistics.

In this paper, we also investigated some approximations to the Beta Median method. In particular, we considered the
approximations of the form $(i-a)/(n+b)$ and showed that if this approximation is to be symmetric
($p_i=1-p_{n+1-i}$), it should have the form $p_i=(i-a)/(n+1-2a)$ (which is also a form suggested by \cite{blom-1958} and \cite{Erto-2013}). \colr{To be close to the exact value at
$p_1,p_n$, for large $n$, we must have $a=\log(e/2)$}. In that case, we showed that it is not a very accurate
approximation for example around the 10th percentile, hence concluding no such approximation of the form
$(i-a)/(n+b)$ would be adequately accurate. Moreover, the approximation of \cite{Erto-2013} allowing $a$ to
vary with $n$ which is exact on $p_1,p_n$ suffers from the same issue. By numerical analysis and by inspecting the limits of $p_1,p_n,$ when $n$ becomes large (Equation \ref{eqn-limit-p1}), we showed that another popular approximation, which assumes $a=1/3,$ while performing better in the middle of the probability index
vector, it fails at the small and large indices, e.g.\;$p_1,p_n$. Fortunately our numerical analysis showed that the
algorithm of \cite{cran-1977}, which is also implemented in C and R to calculate the median of the Beta
distribution, works well across the probability index vector. However this is not routinely used in making the
QQ-plots in R or SPSS and instead approximations of the form $(i-a)/(n+b),$ for some $a,b$ are used
(\cite{castillo-2012}). In summary, if the Weibull method is desired, then this is accurate by letting $a=0,b=1$. However if the
Beta Median method is desired, we recommend using the algorithm of \cite{cran-1977} which
is readily available in R.

Finally, we made a comparison between the plotting positions of Weibull method and Beta Median method, in terms of their difference and ratio. In summary the plotting positions of the two method are always close in terms of difference. They are also close in terms of their ratio toward the center of PIV, but they differ at the two ends of PIV.

\vspace{0.5cm}

\noindent{\bf Acknowledgements}: We would like to thank Prof.\;Jim Zidek, Prof.\;Nhu Le and Prof.\;David Scott for
suggestions which have improved this work. The first author was partially supported by research grants from Japanese Society
for Promotion of Science.


\end{document}